\documentclass[11pt]{amsart}
\usepackage{amsmath,amssymb,color,cite,colonequals,verbatim}
\usepackage{etoolbox}
\apptocmd{\sloppy}{\hbadness 10000\relax}{}{}

\numberwithin{equation}{section}

\newtheorem{thm}[equation]{Theorem}

\newtheorem{prop}[equation]{Proposition}
\newtheorem{lemma}[equation]{Lemma}
\newtheorem{cor}[equation]{Corollary}

\theoremstyle{definition}
\newtheorem{rmk}[equation]{Remark}

\newtheorem{defn}[equation]{Definition}

\newcommand{\F}{\mathbb{F}}
\newcommand{\bP}{\mathbb{P}}

\newcommand{\Z}{\mathbb{Z}}

\DeclareMathOperator{\charp}{char}

\DeclareMathOperator{\GL}{GL}

\newcommand{\mybar}[1]{#1\llap{$\overline{\phantom{\rm#1}}$}}

\usepackage[colorlinks,pagebackref,pdftex,bookmarks=false]{hyperref}

\begin{document}

\title[Complete mappings of certain forms]{Determination of all complete mappings of $\F_{q^2}$ of the form $aX^{3q}+bX^{2q+1}+cX^{q+2}+dX^3$}

\author{Zhiguo Ding}
\address{
  School of Mathematics and Statistics, 
  Hunan Normal University, 
  Changsha 410081, China
  }
\email{ding8191@qq.com}

\author{Wei Xiong}
\address{
  School of Mathematics,
  Hunan University,
  Changsha 410082, China
}
\email{weixiong@amss.ac.cn}

\author{Michael E. Zieve}
\address{
  Department of Mathematics,
  University of Michigan,
  530 Church Street,
  Ann Arbor, MI 48109-1043 USA
}
\email{zieve@umich.edu}
\urladdr{https://dept.math.lsa.umich.edu/$\sim$zieve/}

\thanks{The second author was supported in part by the Natural Science Foundation of Hunan Province of China (No.\ 2020JJ4164).  The third author was supported in part by Simons Travel Grant MPS-TSM-00007931.}

\date{\today}

\begin{abstract}
For each prime power $q$, we determine all polynomials over $\F_{q^2}$ of the form $f(X)\colonequals aX^{3q}+bX^{2q+1}+cX^{q+2}+dX^3$ which induce complete mappings of $\F_{q^2}$, in the sense that each of the functions $x\mapsto f(x)$ and $x\mapsto f(x)+x$ permutes $\F_{q^2}$. 
This is the first result in the literature which classifies the complete mappings among some class of polynomials with arbitrarily large degree over finite fields of arbitrary characteristic.
We also determine all permutation polynomials over $\F_{q^2}$ of the form  $X^{q+2}+bX^q+cX$, and all permutations of $\F_q\times\F_q$ induced by maps of the form $(x,y)\mapsto (x^3-exy^2-ax-by,y^3-cx-dy)$ where either $e=0$ or $3\mid q$.
The latter results add to the small number of results in the literature classifying all permutations induced by maps of prescribed forms.
\end{abstract}

\maketitle


\section{Introduction}

A \emph{complete mapping} of a group $G$ is a permutation $\pi$ of $G$ for which the function $g\mapsto \pi(g)\cdot g$ permutes $G$. Complete mappings were introduced by Mann in his work on constructing orthogonal Latin squares \cite{Mann}. 
They have been used in various ways in cryptography \cite{Mitt1,Mitt2,SV,V}, coding theory \cite{SW}, and in the construction of quasigroups \cite{MM}. By a complete mapping of a field or of a vector space, we mean a complete mapping of the additive group of the relevant object.

There are only a few known classes of complete mappings of finite fields $\F_q$. Most of these come from one of two sources, namely additive homomorphisms of $\F_q$ or functions acting as scalar multiples on each coset of $\F_q^*/H$ for some low-index subgroup $H$ of $\F_q^*$. There are simple conditions determining which functions of these two types are complete mappings.

In this paper we prove the first result classifying all complete mappings over finite fields of arbitrary characteristic among some class of polynomials other than the two just described. 
We give two descriptions of the relevant complete mappings, one providing simple representatives up to a natural equivalence relation, and one providing explicit conditions on the coefficients.

We now define the equivalence relation we will use. We say that $f,g\in\F_{q^2}[X]$ are \emph{$\F_q$-linearly conjugate} if the induced functions on $\F_{q^2}$ satisfy $f=\rho^{-1}\circ g\circ\rho$ for some $\F_q$-vector space automorphism $\rho$ of $\F_{q^2}$.
Equivalently, there exist $a,b\in\F_{q^2}$ such that $a^{q+1}\ne b^{q+1}$ and the polynomials $L(X)\colonequals aX^q+bX$ and $L^{-1}(X)\colonequals (aX^q-b^qX)/(a^{q+1}-b^{q+1})$ satisfy $f(X)\equiv L^{-1}(X)\circ g(X)\circ L(X)\pmod{X^{q^2}-X}$.

Our main result is as follows.

\begin{thm}\label{cm-bb-deg3}
Let $q$ be a power of a prime $p$, and pick any $a,b,c,d\in\F_{q^2}$.
Then $f(X) \colonequals aX^{3q} + bX^{2q+1} + cX^{q+2} + d X^3$ is a complete mapping of\/ $\F_{q^2}$ if and only if one of the following holds:
\renewcommand{\theenumi}{\thethm.\arabic{enumi}}
\renewcommand{\labelenumi}{(\thethm.\arabic{enumi})}
\begin{enumerate}
\item\label{cm1} $f(X)$ is\/ $\F_q$-linearly conjugate to $\gamma X^{q+2}$ for some $\gamma\in\F_{q^2}^*$ with $\gamma^{2q-2}-\gamma^{q-1}+1=0$;
\item\label{cmadd} $q\equiv 0\pmod 3$ and $f(X)=aX^{3q}+dX^3$ where $a^{q+1}\ne d^{q+1}$ and $aX^{3q-1}+dX^2+1$ has no roots in\/ $\F_{q^2}^*$.
\end{enumerate}
\end{thm}

\begin{rmk}
The complete mappings $\gamma X^{q+2}$ in \eqref{cm1} first appeared in \cite[Cor.~2.3]{Z-subfields}. A slightly weaker version of the special case $p=2$ of Theorem~\ref{cm-bb-deg3} was proved in \cite{CDLXXZ}.
\end{rmk}

The complete mappings in \eqref{cmadd} belong to the well-known class of additive polynomials (which are sometimes called linearized polynomials or $p$-polynomials), namely, polynomials of the form $\sum_{i=0}^m a_i X^{p^i}$ where $p$ is the characteristic of $\F_q$.
Such a polynomial $f(X)$ is a complete mapping of $\F_{q^2}$ if and only if both $f(X)$ and $f(X)+X$ have no roots in $\F_{q^2}^*$.
We now determine explicit necessary and sufficient conditions on the coefficients of the non-additive complete mappings $f(X)$ in Theorem~\ref{cm-bb-deg3}.

\begin{thm}\label{cmthm}
For any prime power $q$, and any $a,b,c,d\in\F_{q^2}$, the polynomial $f(X)\colonequals aX^{3q}+bX^{2q+1}+cX^{q+2}+dX^3$ is a complete mapping on\/ $\F_{q^2}$ if and only if $q\not\equiv 1\pmod 3$ and either \eqref{cmadd} holds or one of the following holds:
\renewcommand{\theenumi}{\thethm.\arabic{enumi}}
\renewcommand{\labelenumi}{(\thethm.\arabic{enumi})}
\begin{enumerate}
\item \label{30} $a=b=d=0$ and $c^{2q-2}-c^{q-1}+1=0$; or
\item \label{31} all of these hold:
\begin{itemize}
\item $3ac = b^2\ne 9d^{2q}$,
\item $144a^{q+3} = -(b^2+3d^{2q})^2$,
\item $24a^2d = (b+d^q)(b^2+3d^{2q}) \ne 0$,
\item $24a^2b^q = -(b-3d^q)(b^2+3d^{2q})$; or
\end{itemize}
\item\label{34} $3\mid q$ and all of these hold:
\begin{itemize}
\item $b=d=0$,
\item $c^{q-1}=-1$,
\item $(-a/c)^{(q+1)/2}=-1$; or
\end{itemize}
\item\label{35} $3\mid q$ and all of these hold:
\begin{itemize}
\item $b=0$,
\item $a^{q-1} d^{2q-2} = -1$,
\item $d^{4q+4}+a^4 d^{q+5}$ is a square in\/ $\F_q^*$,
\item $acd^q+a^2d+d^{3q} = 0$; or
\end{itemize}
\item\label{32} $2\mid q$ and all of these hold:
\begin{itemize}
\item $a^{q+1} = c^{2q}+c^{q+1}+c^2$,
\item $ac=b^2$,
\item $d=b^q$,
\item $c\notin\F_q$; or
\end{itemize}
\item\label{33} $q=2$, $a=d$, $b=0$, and $c\in\F_4\setminus\F_2$.
\end{enumerate}
\end{thm}

Our proof of Theorem~\ref{cm-bb-deg3} uses a wide range of tools, including some intricate applications of Hermite's criterion, Weil's bound, results about primitive and doubly transitive permutation groups, and knowledge of all low-degree permutation rational functions. Our proof proceeds by first proving the following classifications of bijections of certain forms, which are of independent interest.

\begin{thm}\label{first}
For any prime power $q$, and any $b,c\in\F_{q^2}$, the polynomial $f(X)\colonequals X^{q+2}+bX^q+cX$ permutes\/ $\F_{q^2}$ if and only if one of the following holds:
\renewcommand{\theenumi}{\thethm.\arabic{enumi}}
\renewcommand{\labelenumi}{(\thethm.\arabic{enumi})}
\begin{enumerate}
\item\label{11} $q\not\equiv 1\pmod 3$, $b=0$, and $c^{q-1}$ is a root of $X^3-X^2+X$;
\item\label{12} $q=2$, $b\ne 0$, and $c=1$.
\end{enumerate}
\end{thm}

\begin{thm}\label{second}
For any prime power $q$, and any $a,b,c,d\in\F_q$, the map $\varphi\colon (x,y)\mapsto (x^3-ax-by, y^3-cx-dy)$ permutes\/ $\F_q\times\F_q$ if and only if one of the following holds:
\renewcommand{\theenumi}{\thethm.\arabic{enumi}}
\renewcommand{\labelenumi}{(\thethm.\arabic{enumi})}
\begin{enumerate}
\item\label{21} $q\not\equiv 1\pmod 3$ and $a=d=bc=0$;
\item\label{22} $q\equiv 0\pmod 3$, $bc=0$, and $a$ and $d$ are nonsquares in\/ $\F_q$;
\item\label{23} $q\equiv 0\pmod 3$, $bc\ne 0$, and no nonzero square in\/ $\F_q$ is a root of the polynomial $X^4 - (a^3+b^2d)X + b^2(ad-bc)$;
\item\label{24} $q=2$, $b=c=1$, and $1\in\{a,d\}$.
\end{enumerate}
\end{thm}

\begin{thm}\label{third}
Let $q$ be a power of $3$, and pick $a,b,c,d,e\in\F_q$ with $e\ne 0$. Then $\varphi\colon (x,y)\mapsto (x^3-exy^2-ax-by, y^3-cx-dy)$ permutes\/ $\F_q\times\F_q$ if and only if $c=0$, $d$ is either zero or a nonsquare, and one of the following holds:
\begin{itemize}
\item $a=0$ and $e$ is a nonsquare; or
\item $q=3$, $a=-1$, and $e=1$.
\end{itemize}
\end{thm}

Our final result relies on the following notion.

\begin{defn}
If $U$ and $V$ are $\F_q$-vector spaces, then a function $f\colon U\to U$ is \emph{$\F_q$-linearly equivalent} to a function $g\colon V\to V$ if $f=\rho\circ g\circ\eta^{-1}$ for some $\F_q$-vector space isomorphisms $\rho$ and $\eta$ from $V$ to $U$.
\end{defn}

\begin{rmk}
It is easy to see that $\F_q$-linear equivalence is an equivalence relation on the union of the sets of functions $\F_{q^2}\to\F_{q^2}$ and $\F_q\times\F_q\to\F_q\times\F_q$, and that $\F_q$-linear equivalence preserves the property of a function being bijective.
\end{rmk}

\begin{rmk}\label{iso}
It is well-known that the $\F_q$-vector space automorphisms of $\F_{q^2}$ are the functions induced by $aX^q+bX$ where $a,b\in\F_{q^2}$ satisfy $a^{q+1}\ne b^{q+1}$.
Likewise, the $\F_q$-vector space isomorphisms $\F_{q^2}\to\F_q\times\F_q$ are the functions $x\mapsto \bigl(ax+(ax)^q,bx+(bx)^q\bigr)$ where $a,b\in\F_{q^2}^*$ satisfy $a^{q-1}\ne b^{q-1}$, and the $\F_q$-vector space isomorphisms $\F_q\times\F_q\to\F_{q^2}$ are $(x,y)\mapsto ax+by$ where $a,b\in\F_{q^2}^*$ satisfy $a^{q-1}\ne b^{q-1}$.
\end{rmk}

\begin{thm}\label{bb-deg3}
Suppose $q$ is a prime power and $a,b,c,d\in\F_{q^2}$. Write $f(X)\colonequals aX^{3q} + bX^{2q+1} + cX^{q+2} + d X^3 \in \F_{q^2}[X]$. Then $f(X)$ permutes\/ $\F_{q^2}$ if and only if $f(X)$ is\/ $\F_q$-linearly equivalent to one of the following:
\renewcommand{\theenumi}{\thethm.\arabic{enumi}}
\renewcommand{\labelenumi}{(\thethm.\arabic{enumi})}
\begin{enumerate}
\item\label{541} $X^{q+2}$, where $q\not\equiv 1\pmod 3$;
\item\label{542} $(X^3,Y^3)$, where $q\not\equiv 1\pmod 3$;
\item\label{543} $(X^3-eXY^2,Y^3)$ for some nonsquare $e\in\F_q^*$, where $q\equiv 0\pmod 3$.
\end{enumerate}
\end{thm}

\begin{rmk}
It is not clear whether Theorem~\ref{bb-deg3} can be used to obtain explicit necessary and sufficient conditions on the coefficients of $f(X)$ which do not involve the coefficients of the additive polynomials involved in the $\F_q$-linear equivalence. However, we will determine such conditions via a different method in a forthcoming paper.
\end{rmk}

The interest of the above results about permutations is that there are only a handful of results determining all bijective functions having specified forms but with arbitrary coefficients and over an arbitrary (possibly square) finite field.
For instance, the most general such result is \cite[Thm.~1.1]{DZ-quad}, which determines
the members of a certain $4$-parameter family of polynomials which permute
$\F_{q^2}$.

This paper is organized as follows. Theorems~\ref{first}, \ref{second}, \ref{third} and \ref{bb-deg3} are proved in Sections 2, 3, 4 and 5, respectively. Then in Section~6 we prove Theorems~\ref{cm-bb-deg3} and \ref{cmthm}.


\section{Bijections induced by $X^{q+2}+bX^q+cX$}

In this section we prove Theorem~\ref{first}.
%
%
Our proof uses the following result of Dickson \cite[\S 14]{Dickson}, which is a generalization to multinomial coefficients of Lucas's theorem on mod $p$ reductions of binomial coefficients.

\begin{lemma}\label{lucas}
Let $p$ be prime, let $m_1,m_2,\dots,m_k$ be nonnegative integers, and write $m\colonequals\sum_{i=1}^k m_i$. Write $m=\sum_{j=0}^\ell b_jp^j$ for integers $b_j$ with $0\le b_j\le p-1$, and for each $i$ write $m_i=\sum_{j=0}^\ell a_{ij} p^j$ for integers $a_{ij}$ with $0\le a_{ij}\le p-1$.
Then the multinomial coefficient $\binom{m} {m_1,m_2,\dots,m_k}$ is coprime to $p$ if and only if for each $j$ with $0\le j\le\ell$ we have $b_j=\sum_{i=1}^k a_{ij}$.
\end{lemma}

By a \emph{term} of the base-$p$ expansion of a nonnegative integer $m$, we mean some $b_j p^j$ as in the above result where $b_j>0$. When we speak of the union of the base-$p$ expansions of multiple integers, we mean the multiset of all terms of all the integers. In case $p=2$, the above result can be written in the following simpler way.


\begin{cor}\label{lucas2}
In the situation of Lemma~\emph{\ref{lucas}}, if $p=2$ then $\binom{m} {m_1,m_2,\dots,m_k}$ is odd if and only if the base-$2$ expansions of $m_i$ and $m_{i'}$ have no common terms whenever $i\ne i'$.
\end{cor}


We also use Hermite's classical criterion for permutation polynomials, as generalized to possibly non-prime finite fields by Dickson \cite[\S 11]{Dickson}:

\begin{lemma}\label{Hermite}
Let $q$ be a prime power, and pick $f(X)\in\F_q[X]$. Then $f(X)$ permutes\/ $\F_q$ if and only if both of the following hold:
\begin{itemize}
\item for each integer $m$ such that $0<m<q-1$ and $\gcd(m,q)=1$, the reduction of $f(X)^m$ mod $X^q-X$ has degree less than $q-1$; and
\item $f(X)$ has exactly one root in\/ $\F_q$.
\end{itemize}
\end{lemma}

When applying Lemma~\ref{Hermite}, it is convenient to observe that if $g(X)$ is the reduction of $f(X)^m$ mod $X^q-X$ then the coefficient of $X^{q-1}$ in $g(X)$ equals the sum of the coefficients of $X^{i(q-1)}$ in $f(X)^m$ for all positive integers $i$.

We now prove Theorem~\ref{first} in case $q>2$ is even and $b\ne 0$. This was shown in \cite{CDLXXZ}, but the short proof below has some new features, so we include it for the reader's convenience.

\begin{lemma}\label{even}
If $q=2^k$ with $k>1$ then $f(X)\colonequals X^{q+2}+bX^q+cX$ does not permute\/ $\F_{q^2}$ for any $b,c\in\F_{q^2}$ with $b\ne 0$.
\end{lemma}

\begin{proof}
We may assume $k>2$, since if $k=2$ then the only term of $f(X)^5$ having degree divisible by $q^2-1$ is $X^{30}$, so Lemma~\ref{Hermite} implies that $f(X)$ does not permute $\F_{q^2}$.
%
%
We will show that the only term of $f(X)^{2q-1}$ having degree divisible by $q^2-1$ is $b^{3q/2} X^{2q^2-2}$. Since $2q-1<q^2-1$, it follows via Lemma~\ref{Hermite} that $f(X)$ does not permute $\F_{q^2}$.

It remains to determine the terms of $f(X)^{2q-1}$ of degree divisible by $q^2-1$. By the multinomial theorem,
\[
f(X)^{2q-1} = \sum_{i=0}^{2q-1} \sum_{j=0}^{2q-1-i} \binom{2q-1}{i,j,2q-1-i-j} b^j c^{2q-1-i-j} X^{i(q+2)+jq+2q-1-i-j}.
\]
The term corresponding to some choice of $i$ and $j$ has degree $i(q+1)+j(q-1)+2q-1$. If this degree is divisible by $q^2-1$ then it is divisible by both $q-1$ and $q+1$, which says that
\[
2i+1\equiv 0\pmod{q-1} \quad\text{ and }\quad -2j-3\equiv 0\pmod{q+1}.
\]
These conditions may be rewritten as
\[
i\equiv \frac{q}2-1\pmod{q-1} \quad\text{ and }\quad j\equiv \frac{q}2-1\pmod{q+1}.
\]
If $i=j=q/2-1$ then the multinomial coefficient $\binom{2q-1}{i,j,2q-1-i-j}$ is divisible by $2$, by Corollary~\ref{lucas2}. Note that if $i=j=q/2-1$ then
\[
i(q+1)+j(q-1)+2q-1 = \Bigl(\frac{q}2-1\Bigr)\cdot 2q + 2q-1 = q^2-1.
\]
If we leave $i$ fixed, and add $q+1$ to $j$, then we increase $i(q+1)+j(q-1)+2q-1$ by $q^2-1$. Likewise, if we leave $j$ fixed, and add $q-1$ to $i$, then we increase $i(q+1)+j(q-1)+2q-1$ by $q^2-1$.
Since $f(X)^{2q-1}$ has degree $(q+2)(2q-1)=2q^2+3q-2<3q^2-3$,
any term of $f(X)^{2q-1}$ with degree divisible by $q^2-1$ must have degree $2q^2-2$, and the coefficient of $X^{2q^2-2}$ is the sum of the contributions from the pairs $(i,j)\in\{(q/2-1,3q/2),(3q/2-2,q/2-1)\}$.  Thus this coefficient is
\[
\binom{2q-1}{\frac{q}2-1,\frac{3q}2} b^{3q/2} + \binom{2q-1}{\frac{3q}2-2,\frac{q}2-1,2} b^{q/2-1} c^2.
\]
By Corollary~\ref{lucas2}, the second multinomial coefficient in the above expression is even since the base-$2$ expansion of $q/2-1$ includes the term $2$ (because $k>2$),
but the first multinomial coefficient in the above expression is odd since the base-$2$ expansions of $q/2-1$ and $3q/2=q+q/2$ have no common terms. 
Thus the only term of $f(X)^{2q-1}$ with degree divisible by $q^2-1$ is $b^{3q/2} X^{2q^2-2}$, so Lemma~\ref{Hermite} implies that $f(X)$ does not permute $\F_{q^2}$.
\end{proof}


Next we prove Theorem~\ref{first} in case $q$ is odd and $b\ne 0$.

\begin{lemma}\label{odd}
If $q$ is an odd prime power then $f(X)\colonequals X^{q+2}+bX^q+cX$ does not permute\/ $\F_{q^2}$ for any $b,c\in\F_{q^2}$ with $b\ne 0$.
\end{lemma}

\begin{proof}
We will show that $f(X)^{q-1}$ is congruent mod $X^{q^2}-X$ to a polynomial of degree $q^2-1$. By Hermite's criterion (Lemma~\ref{Hermite}), it follows that $f(X)$ does not permute $\F_{q^2}$.

By the multinomial theorem,
\[
f(X)^{q-1} = \sum_{i=0}^{q-1} \sum_{j=0}^{q-1-i} \binom{q-1}{i,j,q-1-i-j} b^j c^{q-1-i-j} X^{i(q+2)+jq+q-1-i-j}.
\]
Since $f(X)^{q-1}$ has degree $(q+2)(q-1)=q^2-1+q-1<2q^2-2$, any term of $f(X)^{q-1}$ with degree divisible by $q^2-1$ must have degree equal to $q^2-1$. The summand corresponding to some choice of $i$ and $j$ has degree $i(q+1)+(j+1)(q-1)$.
If this degree equals $q^2-1$ then it is divisible by both $q-1$ and $q+1$, so that $i=r(q-1)/2$ and $j+1=s(q+1)/2$ for some integers $r\ge 0$ and $s>0$.
Conversely, for such $i$ and $j$ we have $i(q+1)+(j+1)(q-1)=(r+s)(q^2-1)/2$, which equals $q^2-1$ if and only if $r+s=2$. Moreover, for such $i$ and $j$ the hypothesis $q-1-i\ge j$ says that
\[
q-1\ge i+j=r(q-1)/2-1+s(q+1)/2=(r+s)(q-1)/2-1+s=q-2+s,
\]
so that $s\le 1$. Thus $i(q+1)+(j+1)(q-1)$ equals $q^2-1$ if and only if $r=s=1$, so that $i=(q-1)/2$ and $j+1=(q+1)/2$. Hence the coefficient of $X^{q^2-1}$ in $f(X)^{q-1}$ is
\[
\binom{q-1}{\frac{q-1}2,\frac{q-1}2} b^{(q-1)/2}.
\]
Writing $q=p^k$ where $p$ is prime, the base-$p$ expansions of $q-1$ and $(q-1)/2$ are $\sum_{i=0}^{k-1} (p-1)p^i$ and $\sum_{i=0}^{k-1}\frac{p-1}2 p^i$, respectively. 
By Lemma~\ref{lucas}, it follows that $\binom{q-1}{\frac{q-1}2,\frac{q-1}2}$ is coprime to $p$, so the coefficient of $X^{q^2-1}$ in $f(X)^{q-1}$ is the product of $b^{(q-1)/2}$ with an element of $\F_q^*$. 
Hence if $b\ne 0$ then the reduction of $f(X)^{q-1}$ mod $(X^{q^2}-X)$ has degree $q^2-1$, which by Lemma~\ref{Hermite} implies that $f(X)$ does not permute $\F_{q^2}$.
\end{proof}


We now prove Theorem~\ref{first}.

\begin{proof}[Proof of Theorem~\ref{first}]
If $b=c=0$ then $f(X)=X^{q+2}$ permutes $\F_{q^2}$ if and only if $\gcd(q+2,q^2-1)=1$; since plainly $\gcd(q+2,q+1)=1$, we have $\gcd(q+2,q^2-1)=\gcd(q+2,q-1)=\gcd(3,q-1)$, so that $f(X)$ permutes $\F_{q^2}$ if and only if $q\not\equiv 1\pmod 3$.
If $b=0\ne c$ then the result is \cite[Cor.~2.3]{Z-subfields}. Henceforth suppose $b\ne 0$. If $q$ is odd then the result is Lemma~\ref{odd}. If $q$ is even and $q>2$ then the result is Lemma~\ref{even}.
Finally, if $q=2$ and $b\ne 0$ then  $f(X)\equiv bX^2+(c+1)X\pmod{X^{q^2}-X}$, so that $f(X)$ induces a homomorphism $\varphi$ from the additive group of $\F_{q^2}$ to itself, and any such $\varphi$ is bijective if and only if its kernel is trivial, or equivalently $c=1$.
\end{proof}


\section{Bijections induced by $(X^3-aX-bY,Y^3-cX-dY)$}

In this section we prove Theorem~\ref{second}. We begin with some terminology and tools used in our proofs. Let $f(X)\in\F_q[X]$ have degree $n>0$, and assume that $\gcd(n,q)=1$. 
If $t$ is transcendental over $\mybar\F_q$ then $f(X)-t$ has no multiple roots in $\mybar{\mybar\F_q(t)}$, since its derivative is a nonzero polynomial in $\mybar\F_q[X]$ and hence has no common roots with $f(X)-t$. 
The \emph{geometric monodromy group} of $f(X)$ is the Galois group of $f(X)-t$ over $\mybar\F_q(t)$, viewed as a group of permutations of the roots of $f(X)-t$. 
Recall that a subgroup $G$ of $S_n$ is \emph{primitive} if the only partitions of $\{1,2,\dots,n\}$ which are preserved by $G$ are the trivial partitions consisting of either one $n$-element set or $n$ one-element sets. 
Also, a subgroup $G$ of $S_n$ is \emph{doubly transitive} if it is transitive on the set of ordered pairs of distinct elements of $\{1,2,\dots,n\}$. 
The following translations between properties of $f(X)$ and properties of $G$ are well known; see for instance \cite[Thms.~6.10--6.12]{LMT}.

\begin{lemma}\label{prim}
Let $f(X)\in\F_q[X]$ have degree $n>0$, where $\gcd(n,q)=1$, and let $G$ be the geometric monodromy group of $f(X)$. Then the following hold:
\begin{itemize}
\item $G$ contains an $n$-cycle;
\item $G$ is primitive if and only if $f(X)$ cannot be written as $g(h(X))$ for any $g,h\in\mybar\F_q[X]$ of degree less than $n$;
\item $G$ is doubly transitive if and only if $(f(X)-f(Y))/(X-Y)$ is irreducible in\/ $\mybar\F_q[X,Y]$.
\end{itemize}
\end{lemma}

The following result is a special case of \cite[Thm.~4]{Fan-Weil}, in light of the discussion about equivalent definitions in \cite[\S 2]{Fan-Weil}.  The main ingredient in its proof is Weil's bound.

\begin{lemma}\label{Weil}
Let $f(X)\in\F_q[X]$ have degree $n>0$. If $(f(X)-f(Y))/(X-Y)$ is irreducible in\/ $\mybar\F_q[X,Y]$ and
\[
q > \Bigl(\frac{(n-2)(n-3)+\sqrt{(n-2)^2(n-3)^2+8n-12}}{2}\Bigr)^2
\]
then $f(X)$ does not permute\/ $\F_q$.
\end{lemma}

Finally, we use the following simple result \cite[\S 18, 22, and 57]{Dickson}.

\begin{lemma}\label{deg3}
Let $q$ be a prime power, $n$ a positive integer, and $a\in\F_q^*$.
Then $X^n$ permutes\/ $\F_q$ if and only if $\gcd(n,q-1)=1$, and $X^3-aX$ permutes\/ $\F_q$ if and only if $3\mid q$ and $a$ is a nonsquare in\/ $\F_q$.
\end{lemma}

With these ingredients in hand, we now prove Theorem~\ref{second}.

\begin{proof}[Proof of Theorem~\ref{second}]
First assume $bc=0$. Since the hypothesis and conclusion are unchanged by interchanging the pairs $(a,b)$ and $(d,c)$ when $bc=0$, we may assume that $b=0$.
Then $\varphi$ is bijective if and only if both $X^3-aX$ and $X^3-dX$ permute $\F_q$, which by Lemma~\ref{deg3} says that either \eqref{21} and \eqref{22} holds.

Henceforth assume $bc\ne 0$. Then the first coordinate of $\varphi(x,y)$ takes value $u$ if and only if $y=(x^3-ax-u)/b$, in which case the second coordinate of $\varphi(x,y)$ is $H_u(x)$ where
\[
H_u(X)\colonequals \Bigl(\frac{X^3-aX-u}b\Bigr)^3 - cX - d\cdot\frac{X^3-aX-u}b.
\]
Therefore $\varphi$ is bijective if and only if $H_u(X)$ permutes $\F_q$ for all $u\in\F_q$. If $q\equiv 0\pmod 3$ then
\[
H_u(X) = b^{-3}X^9 - (a^3b^{-3}+db^{-1})X^3 - (c-dab^{-1})X-u^3b^{-3}+dub^{-1},
\]
so that $H_u(X)-H_u(0)$ induces a homomorphism from the additive group of $\F_q$ to itself, and hence is bijective if and only if its kernel is trivial, which yields \eqref{23}. Henceforth assume $q\not\equiv 0\pmod 3$.

We now prove the result for $q\le 1793$. If $q=2$ then $\varphi(x,y)=\bigl((1+a)x+by,cx+(1+d)y\bigr)$ is a linear transformation of the $\F_2$-vector space $\F_2^2$, and hence is bijective if and only if its determinant is nonzero, which yields \eqref{24}.
If $2<q\le 1793$ (and $3\nmid q$) then we check via Magma that there are no $a,b,c,d\in\F_q$ for which $bc\ne 0$ and every $H_u(X)$ permutes $\F_q$. To speed up this program, we first compose on both sides with scalar multiples in order to reduce to the case that $b=1$ and $a$ is either $0$, $1$, or (for odd $q$) a prescribed nonsquare in $\F_q$.

Henceforth assume $q>1793$. We now show that $H_u(X)$ is not the composition of two degree-$3$ polynomials in $\mybar\F_q[X]$. Suppose to the contrary that $H_u(X)=g(h(X))$ for some $g,h\in\mybar\F_q[X]$ of degree $3$. 
By replacing $g(X)$ and $h(X)$ by $g(\rho(X))$ and $\rho^{-1}(h(X))$ for a suitably chosen degree-$1$ $\rho(X)\in\mybar\F_q[X]$, we may assume that $h(X)$ is monic and $h(0)=0$. 
Equating leading terms in $H_u(X)=g(h(X))$ shows that the leading coefficient of $g(X)$ is $b^{-3}$. Equating terms of degrees $8$ and $7$ shows that $h(X)=X^3-aX$. 
But this is impossible, since plainly $H_u(X)+cX$ is in $\mybar\F_q[X^3-aX]$, so that $H_u(X)$ cannot also be in $\mybar\F_q[X^3-aX]$.

We have shown that $H_u(X)$ is not the composition of two degree-$3$ polynomials in $\mybar\F_q[X]$. Since $H_u(X)$ has degree $9$, it follows that $H_u(X)$ is not the composition of two lower-degree polynomials in $\mybar\F_q[X]$. 
By Lemma~\ref{prim}, the geometric monodromy group of $H_u(X)$ is a primitive subgroup of $S_9$ which contains a $9$-cycle. But any such group is doubly transitive: this can be shown by directly checking all such groups, or alternately it is a special case of Schur's theorem (e.g., cf.\ \cite[Thm.~6.5]{LMT}). 
By Lemma~\ref{prim}, it follows that $(H_u(X)-H_u(Y))/(X-Y)$ is  irreducible in $\mybar\F_q[X,Y]$. Since $q>1793$, Lemma~\ref{Weil} implies that $H_u(X)$ does not permute $\F_q$, which concludes the proof.
%
%
\end{proof}


\section{Bijections induced by $(X^3-eXY^2-aX-bY,Y^3-cX-dY)$}

In this section we prove Theorem~\ref{third}. We begin with the following elementary lemmas.

\begin{lemma}\label{int}
Write $Q\colonequals 3^\ell$ for some $\ell\ge 3$. Let $n_1,n_2,n_3,n_5,n_9$ be nonnegative integers for which $n_1+2n_2+3n_3+5n_5+9n_9=Q-1$ and the union of the base-$3$ expansions of the $n_j$'s consists of one copy of each $3^i$ with $1\le i\le\ell-2$ along with some partition of $2$. Then the base-$3$ expansion of $n_5$ contains $Q/9$, and the base-$3$ expansion of $n_9$ contains $Q/27$.
\end{lemma}

\begin{proof}
The sum of the $n_j$'s is $1+(Q/3-1)/2=(Q+3)/6$. Since $Q-1\ge 9n_9$, we have $n_9<Q/9$. It follows that the base-$3$ expansion of $n_5$ includes $Q/9$, since otherwise we obtain the contradiction
\[
Q-1 \le 2\frac{Q}9+9\Bigl(\frac{Q+3}6-\frac{Q}9\Bigr) = \frac{13Q+81}{18}<Q-1.
\]
%
%
Next, if $Q>27$ then the base-$3$ expansion of $n_9$ includes $Q/27$, since otherwise we obtain the contradiction
\begin{align*}
Q-1 &\le 5\Bigl(\frac{Q}9+\frac{Q}{27}\Bigr)+9\Bigl(\frac{Q+3}6-\frac{Q}9-\frac{Q}{27}\Bigr) \\ 
&= \frac{49Q+243}{54} \\ 
&< Q-1.
\end{align*}
%
%
Finally, if $Q=27$ then $Q-1-5Q/9=11$ equals $n_1+2n_2+3n_3+5(n_5-3)+9n_9$ where $n_1,n_2,n_3,n_5-3,n_9$ are nonnegative integers whose sum is $2$. It follows that $n_2=n_9=1$, so that $n_9=Q/27$.
\end{proof}

\begin{lemma}\label{int2}
Write $Q\colonequals 3^\ell$ with $\ell\ge 5$. Let $n_1,n_2,n_5,n_9$ be nonnegative integers for which $n_1+2n_2+5n_5+9n_9=Q-1$ and the union of the base-$3$ expansions of the $n_j$'s consists of one copy of each $3^i$ with $0\le i\le \ell-2$ and $i\ne 2$, along with either one copy of $2\cdot 9$ or two copies of $9$. Then the base-$3$ expansion of $n_5$ contains $Q/9$, and if $\ell\ge 6$ then the base-$3$ expansion of $n_9$ contains $Q/27$.
\end{lemma}

We omit the proof of Lemma~\ref{int2}, since it is nearly identical to that of Lemma~\ref{int}.

We now use the above lemmas to prove the following result, which is of some interest for its own sake.

\begin{prop}\label{newlem}
Let $q=3^k$ for some $k\ge 1$, and pick $a,b,c,d\in\F_q$ with $ac\ne 0$. Then $f(X)\colonequals X^9+aX^5+bX^3+cX^2+dX$ does not permute\/ $\F_q$.
\end{prop}

\begin{proof}
Suppose to the contrary that $f(X)$ permutes $\F_q$. We first apply Hermite's criterion (Lemma~\ref{Hermite}) with exponent $(q+3)/6$, noting that this exponent is a positive integer which is less than $q-1$.
If $q=3$ then we obtain the contradiction $c=0$.
If $q=9$ then we obtain $-ab=0$, so that $b=0$; 
then Lemma~\ref{Hermite} with exponents $4$ and $5$ yield $a^3(d+1)+a(d+1)^3+c^4=0$ and $c^3(a^2+(d+1)^2)=0$,
%
%
so that $a^2=-(d+1)^2$ and thus $c^4=0$, contradiction. Henceforth assume $q>9$. Then $g(X)\colonequals f(X)^{(q+3)/6}$ has degree $9(q+3)/6$, which is less than $2q-2$.
Thus Lemma~\ref{Hermite} implies that the coefficient of $X^{q-1}$ in $g(X)$ is zero. Letting $S$ be the set of all tuples of nonnegative integers $(n_1,n_2,n_3,n_5,n_9)$ such that $n_1+n_2+n_3+n_5+n_9=(q+3)/6$ and $n_1+2n_2+3n_3+5n_5+9n_9=q-1$, it follows that
\begin{equation}\label{hermnew}
\sum_{(n_1,n_2,n_3,n_5,n_9)\in S}\binom{\displaystyle{\frac{q+3}6}}{n_1,n_2,n_3,n_5,n_9} a^{n_5} b^{n_3} c^{n_2} d^{n_1} = 0.
\end{equation}
Let $T$ be the set of tuples $(n_1,n_2,n_3,n_5,n_9)$ in $S$ for which $\binom{(q+3)/6}{n_1,n_2,n_3,n_5,n_9}$ is coprime to $3$. Note that $(q+3)/6=1+\sum_{i=0}^{k-2} 3^i$.
By Lemma~\ref{lucas}, $T$ consists of the tuples $(n_1,n_2,n_3,n_5,n_9)$ in $S$ for which the union of the base-$3$ expansions of the $n_j$'s consists of a single copy of each $3^i$ with $1\le i\le k-2$ together with some partition of $2$. 
Pick any $(n_1,n_2,n_3,n_5,n_9)\in T$. We prove by induction on $i$ that if $0<2i\le k-1$ then $3^{k-2i}$ is a term in the base-$3$ expansion of $n_5$ and $3^{k-1-2i}$ is a term in the base-$3$ expansion of $n_9$. 
The base case $i=1$ follows from Lemma~\ref{int}. Now suppose $2<2i\le k-1$. By the inductive hypothesis, we know that if $1\le j\le i-1$ then the base-$3$ expansion of $n_5$ includes $3^{k-2j}$ and the base-$3$ expansion of $n_9$ includes $3^{k-1-2j}$. 
Let $n_z'$ be the least nonnegative residue of $n_z$ mod $3^{k-2i+1}$. Then $n_1'+2n_2'+3n_3'+5n_5'+9n_9'=q-1-5\sum_{j=1}^{i-1} 3^{k-2j}-9\sum_{j=1}^{i-1} 3^{k-1-2j} = q/3^{2i-2} - 1$,
%
%
and the union of the base-$3$ expansions of the various $n_z'$ consists of one copy of each $3^{\ell}$ with $1\le \ell\le k-2i$ along with some partition of $2$. Also the hypothesis $2i\le k-1$ says that $3\le k-2i+2$. 
Hence by Lemma~\ref{int} we conclude that the base-$3$ expansions of $n_5$ and $n_9$ include $3^{k-2i}$ and $3^{k-1-2i}$, respectively. This concludes the induction.

If $k$ is odd then, writing $n_5'$ and $n_9'$ for the least nonnegative residues of $n_5$ and $n_9$ mod $3$, we have $n_1+2n_2+3n_3+5n_5'+9n_9'=11$ and $n_1+n_2+n_3+n_5'+n_9'=2$. Thus $n_2=n_9'=1$ and $n_1=n_3=n_5'=0$. 
We have shown that $n_1=n_3=0$, $n_2=1$, $n_5=\sum_{i=1}^{(k-1)/2} 3^{2i-1}$, and $n_9=\sum_{i=0}^{(k-3)/2} 3^{2i}$. Conversely, it is easy to check that these $n_j$'s indeed yield a tuple in $T$.
Thus \eqref{hermnew} implies that $a^{n_5}c=0$, contradiction.

Henceforth suppose $k$ is even. Writing $n_5'$ and $n_9'$ for the least nonnegative residues of $n_5$ and $n_9$ mod $3$, we have $n_1+2n_2+3n_3+5n_5'+9n_9'=8$ and $n_1+n_2+n_3+n_5'+n_9'=2$. 
Thus $n_3=n_5'=1$ and $n_1=n_2=n_9'=0$, so $n_1=n_2=0$, $n_3=1$, $n_5=\sum_{i=0}^{k/2-1} 3^{2i}$, and $n_9=\sum_{i=1}^{k/2-1} 3^{2i-1}$.
Conversely, it is easy to check that this tuple of $n_j$'s yields an element of $T$.
Thus \eqref{hermnew} implies that $a^{n_5}b=0$, so that $b=0$.

Next, in case $k>2$ is even and $b=0$, we apply Lemma~\ref{Hermite} with exponent $(q+51)/6$, noting that this is a positive integer less than $q-1$.
In case $k=4$, this yields the contradiction $-a^{12}c^{10}=0$,
%
%
so we assume henceforth that $k\ge 6$. Then $g(X)\colonequals f(X)^{(q+51)/6}$ has degree less than $2q-2$,
so Lemma~\ref{Hermite} implies that the coefficient of $X^{q-1}$ in $g(X)$ is zero. 
Letting $S$ be the set of all tuples of nonnegative integers $(n_1,n_2,n_5,n_9)$ such that $n_1+n_2+n_5+n_9=(q+51)/6$ and $n_1+2n_2+5n_5+9n_9=q-1$, it follows that
\begin{equation}\label{hermnew2}
\sum_{(n_1,n_2,n_5,n_9)\in S} \binom{\displaystyle{\frac{q+51}6}}{n_1,n_2,n_5,n_9} a^{n_5} c^{n_2} d^{n_1} = 0.
\end{equation}
Let $T$ be the set of tuples $(n_1,n_2,n_5,n_9)$ in $S$ for which $\binom{(q+51)/6}{n_1,n_2,n_5,n_9}$ is coprime to $3$. Note that $(q+51)/6=9+\sum_{i=0}^{k-2}3^i$. 
By Lemma~\ref{lucas}, $T$ consists of the tuples $(n_1,n_2,n_5,n_9)$ in $S$ for which the union of the base-$3$ expansions of the $n_j$'s consists of a single $3^i$ for each $i$ with $0\le i\le k-2$ and $i\ne 2$, along with either one copy of $2\cdot 9$ or two copies of $9$. 
Now suppose that $(n_1,n_2,n_5,n_9)\in T$. Arguing as above, but now using Lemma~\ref{int2}, we conclude that the base-$3$ expansion of $n_5$ includes $3^{k-2},3^{k-4},\dots,3^4$, and the base-$3$ expansion of $n_9$ includes $3^{k-3},3^{k-5},\dots,3^5$.
%
%
Writing $n_j'$ for the least nonnegative residue of $n_j$ mod $81$, we have $n_1'+2n_2'+5n_5'+9n_9'=323$ 
%
%
and $n_1'+n_2'+n_5'+n_9'=49$,
where we know that for each $i\in\{0,1,3\}$ there is a unique $j$ for which the base-$3$ expansion of $n_j'$ includes $3^i$, and moreover there is no $j$ for which the base$-3$ expansion of $n_j'$ includes $2\cdot 3^i$. The unique solution is $(n_1',n_2',n_5',n_9')=(0,10,12,27)$.
%
%
Thus $n_1=0$, $n_2=10$, $n_5=3+\sum_{i=1}^{k/2-1} 3^{2i}$, and $n_9=\sum_{i=2}^{k/2-1} 3^{2i-1}$. 
Conversely, it is easy to check that these $n_j$'s yield an element of $T$,
so that \eqref{hermnew2} gives the contradiction $a^{10} c^{n_5}=0$.
\end{proof}

We now prove Theorem~\ref{third}.

\begin{proof}[Proof of Theorem~\ref{third}]
First suppose $c=0$. The projection of the image of $\varphi$ onto the second coordinate is surjective if and only if $Y^3-dY$ permutes $\F_q$, which holds if and only if either $d=0$ or $d$ is a nonsquare in $\F_q$. 
Suppose that $Y^3-dY$ permutes $\F_q$.
Then $\varphi$ is bijective if and only if, for each $y\in\F_q$, the polynomial $X^3-(ey^2+a)X$ permutes $\F_q$.
By Lemma~\ref{deg3}, this says that $ey^2+a$ is either zero or a nonsquare in $\F_q$ for each $y\in\F_q$. 
If $a=0$ then this condition holds if and only if $e$ is a nonsquare.
Now assume $a\ne 0$.
If $q=3$ then $\{ey^2+a:y\in\F_q\}=\{a,e+a\}$, so that $\varphi$ is bijective if and only if both $a$ and $e+a$ are in $\{0,-1\}$, which holds if and only if $a=-1$ and $e=1$. 
We show now that $\varphi$ is not bijective when $a\ne 0=c$ and $q>3$.
Suppose otherwise.
Since $a\ne 0$, the curve $X^2=eY^2+a$ is irreducible over $\mybar\F_q$. 
The closure $C$ of this curve in $\bP^2$ is nonsingular, and has at most two $\F_q$-rational points with $Y=\infty$.
For any $x,y\in\F_q$ with $x^2=ey^2+a$, the hypothesis that $ey^2+a$ is either zero or a nonsquare implies that $x=0$.
Since at most two elements $y\in\F_q$ satisfy $ey^2+a=0$, we conclude that $C$ has at most four $\F_q$-rational points.
But $C$ has genus $0$, so it has $q+1$ $\F_q$-rational points, which is impossible since $q>3$.

Henceforth suppose $c\ne 0$.
For any $v\in\F_q$, the second coordinate of $\varphi(x,y)$ equals $v$ if and only if $y^3-cx-dy=v$, or equivalently $x=(y^3-dy-v)/c$. 
Thus if the second coordinate of $\varphi(x,y)$ equals $v$ then the first coordinate of $\varphi(x,y)$ is
\[
c^{-3}(y^3-dy-v)^3 - ec^{-1}(y^3-dy-v)y^2-ac^{-1}(y^3-dy-v)-by.
\]
Therefore $\varphi$ is bijective if and only if, for each $v\in\F_q$, the polynomial
\[
H_v(Y)\colonequals (Y^3-dY-v)^3 - ec^2(Y^3-dY-v)Y^2 - ac^2(Y^3-dY-v)-bc^3Y
\]
permutes $\F_q$.
We compute
\[
H_v(Y)-H_v(0) = Y^9 - ec^2Y^5 + (edc^2-ac^2-d^3)Y^3 + ec^2vY^2 + (adc^2-bc^3)Y.
\]
Since $c$ and $e$ are nonzero, Proposition~\ref{newlem} implies that $H_v(Y)$ does not permute $\F_q$ for any $v\in\F_q^*$.
Therefore $\varphi$ is not bijective.
\end{proof}


\section{Bijections induced by $aX^{3q}+bX^{2q+1}+cX^{q+2}+dX^3$}

In this section we prove Theorem~\ref{bb-deg3}. We first present the notation, terminology, and previous results we will use.
\begin{itemize}
\item $q$ is a fixed prime power;
\item $\mu_{q+1}$ denotes the set of $(q+1)$-th roots of unity in $\F_{q^2}$;
\item $\bP^1(\F_q) \colonequals \F_q\cup\{\infty\}$ is the set of $\F_q$-rational points on $\bP^1$;
\item for any $g(X)\in\mybar\F_q(X)$, we write $g^{(q)}(X)$ for the rational function obtained from $g(X)$ by replacing each coefficient by its $q$-th power;
\item the \emph{degree} of a nonzero rational function $g(X)$ is the maximum of the degrees of $N(X)$ and $D(X)$, for any prescribed choice of coprime polynomials $N(X)$ and $D(X)$ such that $g(X)=N(X)/D(X)$.
\end{itemize}

The following result is a special case of \cite[Lemma~2.1]{Z-lem}.

\begin{lemma}\label{old}
Write $f(X)=X^r B(X^{q-1})$ where $r$ is a positive integer, $q$ is a prime power, and $B(X)\in\F_{q^2}[X]$. Then $f(X)$ permutes\/ $\F_{q^2}$ if and only if $\gcd(r,q-1)=1$ and $g_0(X)\colonequals X^r B(X)^{q-1}$ permutes $\mu_{q+1}$.
\end{lemma}

The following result encodes a procedure introduced in \cite{Z-Redei}, which is spelled out in \cite[Lemma~2.2]{Z-x}.

\begin{lemma}\label{lemx}
Let $q$ be a prime power, and write $g_0(X)=X^r B(X)^{q-1}$ where $r\in\Z$ and $B(X)\in\F_{q^2}[X]$. Then $g_0(X)$ permutes $\mu_{q+1}$ if and only if $B(X)$ has no roots in $\mu_{q+1}$ and $g(X)\colonequals X^r B^{(q)}(1/X)/B(X)$ permutes $\mu_{q+1}$.
\end{lemma}

The next two results are immediate consequences of \cite[Lemmas~2.1 and 3.1]{Z-Redei}.

\begin{lemma}\label{deg1mu}
The degree-one rational functions in\/ $\mybar\F_q(X)$ which permute $\mu_{q+1}$ are precisely the functions $\rho(X)=(\beta^qX+\alpha^q)/(\alpha X+\beta)$ where $\alpha,\beta\in\F_{q^2}$ satisfy $\alpha^{q+1}\ne \beta^{q+1}$.
\end{lemma}

\begin{lemma}\label{mu}
The degree-one rational functions in\/ $\mybar\F_q(X)$ which map\/ $\bP^1(\F_q)$ bijectively onto $\mu_{q+1}$ are $(\beta^q X+\alpha^q)/(\beta X+\alpha)$ with $\beta\in\F_{q^2}^*$ and $\alpha\in\F_{q^2}$ such that $\alpha/\beta\notin\F_q$.
The degree-one rational functions in\/ $\mybar\F_q(X)$ which map $\mu_{q+1}$ bijectively onto\/ $\bP^1(\F_q)$ are $(\gamma X+\gamma^q)/(\delta X+\delta^q)$ with $\delta\in\F_{q^2}^*$ and $\gamma\in\F_{q^2}$ such that $\gamma/\delta\notin\F_q$. 
\end{lemma}

The following result is \cite[Thm.~1.3]{DZ-AA}.

\begin{lemma}\label{pp3}
A degree-three $h(X)\in\F_q(X)$ permutes\/ $\bP^1(\F_q)$ if and only if there exist degree-one $\rho,\eta\in\F_q(X)$ for which $\rho(h(\eta(X)))$ is one of the following:
\renewcommand{\theenumi}{\thethm.\arabic{enumi}}
\renewcommand{\labelenumi}{(\thethm.\arabic{enumi})}
\begin{enumerate}
\item\label{deg31} $X^3$ where $q\not\equiv 1\pmod 3$;
\item $\nu^{-1}\circ X^3\circ\nu$ where $q\equiv 1\pmod{3}$ and for some $\delta\in\F_{q^2}\setminus\F_q$ we have $\nu(X)=(X-\delta^q)/(X-\delta)$ and $\nu^{-1}(X)=(\delta X-\delta^q)/(X-1)$;
\item\label{deg33} $X^3-\alpha X$ where $3\mid q$ and $\alpha$ is a nonsquare in\/ $\F_q$.
\end{enumerate}
\end{lemma}

\begin{proof}[Proof of Theorem~\ref{bb-deg3}]
The ``if'' implication follows from Lemma~\ref{deg3}.
It remains to prove the ``only if'' implication.
Thus, we assume in what follows that $f(X)$ permutes $\F_{q^2}$, so in particular $a,b,c,d$ are not all zero.
Write $B(X)\colonequals aX^3+bX^2+cX+d$, and write $\widehat B(X)\colonequals d^qX^3+c^qX^2+b^qX+a^q$, so that $\widehat B(X)=X^3 B^{(q)}(X^{-1})$. Let $C(X)\colonequals\gcd(B(X),\widehat B(X))$, where we may assume that $C(X)$ is monic. Write $g(X)\colonequals \widehat B(X) / B(X)$. 
By Lemmas~\ref{old} and \ref{lemx}, the hypothesis that $f(X)$ permutes $\F_{q^2}$ implies that $q\not\equiv 1\pmod 3$, $B(X)$ has no roots in $\mu_{q+1}$, and $g(X)$ permutes $\mu_{q+1}$. Since $B(X)$ has no roots in $\mu_{q+1}$, also $C(X)$ has no roots in $\mu_{q+1}$.

First suppose that $C(X)$ has a root $\gamma\in\mybar\F_q^*$. Then $0=B(\gamma)^q=B^{(q)}(\gamma^q)=\gamma^{3q}\widehat B(\gamma^{-q})$, so that $\widehat B(\gamma^{-q})=0$, and likewise $0=\widehat B(\gamma)^q=\gamma^{3q} B(\gamma^{-q})$ implies $B(\gamma^{-q})=0$. 
Thus $C(\gamma^{-q})=0$, and we must have $\gamma^{-q}\ne\gamma$ since $C(X)$ has no roots in $\mu_{q+1}$. Hence $\gamma$ and $\gamma^{-q}$ are distinct roots of $C(X)$. 
Since $\deg(g)=\max(\deg(B),\deg(\widehat B))-\deg(C)\le 3-\deg(C)$, and $g(X)$ must be nonconstant because it permutes $\mu_{q+1}$, we conclude that $C(X)=(X-\gamma)(X-\gamma^{-q})$.

Now suppose that $\deg(g)<3$. We claim that $B(X) = (\alpha X + \beta) (X - \gamma) (\gamma^q X - 1)$ for some $\alpha,\beta,\gamma\in\F_{q^2}$ with $\gamma\notin\mu_{q+1}$. If $\{a,d\}=\{0\}$ then the claim holds with $\alpha\colonequals -b$, $\beta\colonequals -c$, and $\gamma\colonequals 0$. 
If $\{a,d\}\ne\{0\}$ then $C(0)\ne 0$ and $\max(\deg(B)\deg(\widehat B))=3$, so that $C(X)$ has a root $\gamma\in\mybar\F_q^*$ and thus the previous paragraph yields $C(X)=(X-\gamma)(X-\gamma^{-q})$ and $\gamma\notin\mu_{q+1}$, which implies the claim since $C(X)$ divides $B(X)$. 
Thus the claim holds in every case. It is easy to check that $f(X)\equiv M(X)\circ X^{q+2}\circ L(X)\pmod{X^{q^2}-X}$ where $L(X)\colonequals X^q-\gamma X$ and $M(X)\colonequals \bigl((\alpha\gamma+\beta)X^q+(\alpha+\gamma^q \beta)X\bigr)/(\gamma^{q+1}-1)$.
%
%
Plainly $L(X)$ and $M(X)$ are $\F_q$-linear maps $\F_{q^2}\to\F_{q^2}$, and they are bijective since $f(X)$ is bijective. Thus \eqref{541} holds.

Henceforth suppose that $\deg(g)=3$. Pick any $z\in\F_{q^2}\setminus\F_q$. By Lemma~\ref{mu}, $\theta(X)\colonequals (zX-z^q)/(X-1)$ defines a bijection from $\mu_{q+1}$ onto $\bP^1(\F_q)$, and $\theta^{-1}(X)\colonequals (X-z^q)/(X-z)$ defines the inverse bijection from $\bP^1(\F_q)$ onto $\mu_{q+1}$. 
Write $h(X)\colonequals \theta(X) \circ g(X) \circ \theta^{-1}(X)$, so that $\deg(h)=3$ and $h(X)$ permutes $\bP^1(\F_q)$. It is easy to check that $h^{(q)}(X)=h(X)$, so that $h(X)\in\F_q(X)$.
%
%
Since $q\not\equiv 1\pmod 3$, by Lemma~\ref{pp3} there exist degree-one $\widehat\rho,\widehat\eta\in\F_q(X)$ for which $\widehat\rho(h(\widehat\eta(X)))$ is either \eqref{deg31} or \eqref{deg33}. 
It follows that there exist degree-one $\rho,\eta\in\F_{q^2}(X)$ such that $\rho^{-1}(g(\eta^{-1}(X)))$ is either \eqref{deg31} or \eqref{deg33}, where in addition $\rho(\bP^1(\F_q))=\mu_{q+1}$ and $\eta(\mu_{q+1})=\bP^1(\F_q)$. 
By Lemma~\ref{mu}, we have $\rho(X)=(\beta^q X+\alpha^q)/(\beta X+\alpha)$ and $\eta(X)=(\gamma X+\gamma^q)/(\delta X+\delta^q)$ for some $\alpha,\beta,\gamma,\delta\in\F_{q^2}$ such that $\beta,\delta\ne 0$ and $\alpha\beta^{-1},\gamma\delta^{-1}\notin\F_q$. 
Let $\widetilde\rho\colon\F_q\times\F_q\to\F_{q^2}$ and $\widetilde\eta\colon\F_{q^2}\to\F_q\times\F_q$ map $\widetilde\rho\colon (x,y)\mapsto \beta x+\alpha y$ and $\widetilde\eta\colon x\mapsto (\gamma x^q+\gamma^q x, \delta x^q+\delta^q x)$, so that $\widetilde\rho$ and $\widetilde\eta$ are $\F_q$-linear. 
We have $g(X)=\rho(X)\circ (X^3-eX)\circ\eta(X)$, where if $e\ne 0$ then $3\mid q$ and $e$ is a nonsquare in $\F_q$. Then it is easy to check that there is some $\varepsilon\in\F_{q^2}^*$ for which $\varepsilon\cdot f(X)$ induces the same function on $\F_{q^2}$ as does $\widetilde\rho\circ (X^3-eXY^2,Y^3)\circ \widetilde\eta$.
%
%
Since bijectivity of $f(X)$ implies bijectivity of $\varepsilon^{-1}\widetilde\rho$ and $\widetilde\eta$, this yields \eqref{542} if $e=0$ and \eqref{543} if $e\ne 0$.
\end{proof}


\section{Complete mappings}

In this section we prove Theorems~\ref{cm-bb-deg3} and \ref{cmthm}.

\begin{proof}[Proof of Theorem~\ref{cm-bb-deg3}]
We first prove the ``if'' implication. If $\gamma\in\F_{q^2}^*$ satisfies $\gamma^{2q-2}-\gamma^{q-1}+1=0$ then $q\not\equiv 1\pmod 3$, and $\gamma X^{q+2}$ is a complete mapping of $\F_{q^2}$ by \cite[Cor.~3.4]{Z-subfields}. 
Since $\F_q$-linear conjugacy preserves the complete mapping property, it follow that the polynomials $f(X)$ in \eqref{cm1} are complete mappings. 
If \eqref{cmadd} holds then both $f(X)$ and $f(X)+X$ induce homomorphisms from the additive group of $\F_{q^2}$ to itself, so that they permute $\F_{q^2}$ if and only if these homomorphisms have trivial kernel. 
The kernel of $f(X)$ is trivial because $a^{q+1}\ne d^{q+1}$, and the kernel of $f(X)+X$ is trivial because $(f(X)+X)/X$ has no roots in $\F_{q^2}^*$ by hypothesis. Thus $f(X)$ is a complete mapping of $\F_{q^2}$.

It remains to prove the ``only if'' implication. Henceforth we suppose that $f(X)$ is a complete mapping of $\F_{q^2}$. In particular, $f(X)$ permutes $\F_{q^2}$, so Theorem~\ref{bb-deg3} implies that $f(X)$ if $\F_q$-linearly equivalent to one of \eqref{541}--\eqref{543}.

First suppose that $f(X)$ is \eqref{541} up to $\F_q$-linear equivalence. Thus $q\not\equiv 1\pmod 3$ and $\rho\circ f(X)\circ\eta = X^{q+2}$ as maps on $\F_{q^2}$ for some automorphisms $\rho$ and $\eta$ of $\F_{q^2}$ as an $\F_q$-vector space.
Since $\rho\circ\eta$ is an $\F_q$-vector space automorphism of $\F_{q^2}$, there are $\alpha,\beta\in\F_{q^2}$ with $\alpha^{q+1}\ne \beta^{q+1}$ such that $\rho\circ\eta=\alpha X^q+\beta X$ as maps on $\F_{q^2}$. 
It follows that $\rho \circ (f(X)+X) \circ \eta  = X^{q+2} + \alpha X^q + \beta X$ as maps on $\F_{q^2}$. 
Since $f(X)+X$ permutes $\F_{q^2}$, also $X^{q+2} + \alpha X^q + \beta X$ permutes $\F_{q^2}$, which by Theorem~\ref{first} implies that one of the following holds:
\begin{enumerate}
\item $q\not\equiv 1\pmod 3$, $\alpha=0$, and $\beta^{q-1}$ is a root of $X^3-X^2+X$;
\item $q=2$, $\alpha\ne 0$, and $\beta=1$.
\end{enumerate}
Since $\alpha^{q+1}\ne\beta^{q+1}$, it follows that (1) holds and $\beta\ne 0$. Thus $\beta^{q-1}$ is a root of $X^2-X+1$. Since $\rho=\beta\eta^{-1}$, it follows that $\eta^{-1}\circ f(X)\circ\eta = \beta^{-1} X^{q+2}$ as maps on $\F_{q^2}$, which gives \eqref{cm1}.  

Next suppose that $f(X)$ is \eqref{542} up to $\F_q$-linear equivalence. Thus $q\not\equiv 1\pmod 3$ and $\rho\circ f(X)\circ\eta^{-1} = (X^3,Y^3)$ as maps on $\F_q\times\F_q$ for some $\F_q$-vector space isomorphisms $\rho$ and $\eta$ from $\F_{q^2}$ to $\F_q\times\F_q$. 
Since $\rho\circ\eta^{-1}$ is an $\F_q$-vector space automorphism of $\F_q\times\F_q$, there exists $\Bigl[ \begin{matrix} \alpha & \beta \\ \gamma & \delta \end{matrix} \Bigr] \in \GL_2(\F_q)$ such that $\rho\circ\eta^{-1}$ sends $(x,y)$ to $(\alpha x+\beta y,\gamma x+\delta y)$ for any $x,y\in\F_q$. 
It follows that $\rho \circ (f(X)+X) \circ \eta^{-1} = (X^3+\alpha X+\beta Y,Y^3+\gamma X+\delta Y)$ as maps on $\F_q\times\F_q$. 
Since $f(X)+X$ permutes $\F_{q^2}$, also $(X^3+\alpha X+\beta Y,Y^3+\gamma X+\delta Y)$ permutes $\F_q\times\F_q$, so Theorem~\ref{second} implies that one of the following holds:
\begin{enumerate}
\item $q\equiv 0\pmod 3$, $\beta\gamma=0$, and $-\alpha$ and $-\delta$ are nonsquares in $\F_q$;
\item $q\equiv 0\pmod 3$, $\beta\gamma\ne 0$, and no square in $\F_q$ is a root of the polynomial $X^4 + (\alpha^3+\beta^2\delta)X + \beta^2(\alpha\delta-\beta\gamma)$;
\item $q=2$, $\beta=\gamma=1$, and $\alpha+\delta=1$.
\end{enumerate}
If (3) holds then it is easy to check that $f(X)$ is $\F_q$-linearly conjugate to $wX$ with $w\in\F_4\setminus\F_2$, so that \eqref{cm1} holds.
%
%
Henceforth we suppose that either (1) or (2) holds. Then $q\equiv 0\pmod 3$, so the map $\rho^{-1}\circ(X^3,Y^3)\circ\eta$ on $\F_{q^2}$ is the cube of an $\F_q$-linear automorphism of $\F_{q^2}$, and hence is induced by a polynomial whose terms have degrees in $\{3,3q\}$. 
Thus $f(X)=aX^{3q}+dX^3$. Since $f(X)$ permutes $\F_{q^2}$, we have $a^{q+1}\ne d^{q+1}$. Since $f(X)+X$ permutes $\F_{q^2}$, the polynomial $(f(X)+X)/X=aX^{3q-1}+dX^2+1$ has no roots in $\F_{q^2}^*$, so that \eqref{cmadd} holds.

Finally, suppose that $f(X)$ is \eqref{543} up to $\F_q$-linear equivalence. Thus $q\equiv 0\pmod 3$ and $\rho\circ f(X)\circ\eta^{-1} = (X^3-eXY^2,Y^3)$ as maps on $\F_q\times\F_q$, where $e$ is a nonsquare in $\F_q^*$ and $\rho$ and $\eta$ are $\F_q$-vector space isomorphisms $\F_{q^2}\to \F_q\times\F_q$. 
Since $\rho\circ\eta^{-1}$ is an $\F_q$-vector space automorphism of $\F_q\times\F_q$, there exists $\Bigl[ \begin{matrix} \alpha & \beta \\ \gamma & \delta \end{matrix} \Bigr] \in \GL_2(\F_q)$ such that $\rho\circ\eta^{-1}$ sends $(x,y)$ to $(\alpha x+\beta y,\gamma x+\delta y)$ for any $x,y\in\F_q$.
It follows that $\rho \circ (f(X)+X) \circ \eta^{-1} = (X^3-eXY^2+\alpha X+\beta Y,Y^3+\gamma X+\delta Y)$ as maps on $\F_q\times\F_q$. 
Since $f(X)+X$ permutes $\F_{q^2}$, also $(X^3-eXY^2+\alpha X+\beta Y,Y^3+\gamma X+\delta Y)$ permutes $\F_q\times\F_q$, which is impossible by Theorem~\ref{third} since $e$ is a nonsquare in $\F_q$ and $\alpha\delta\ne\beta\gamma$.
\end{proof}

We conclude this paper by proving Theorem~\ref{cmthm}.

\begin{proof}[Proof of Theorem~\ref{cmthm}]
First we prove the result when $q=2$. If $q=2$ then $f(X)\equiv (a+d)X^3+bX^2+cX\pmod{X^4+X}$; since the only permutation polynomials over $\F_4$ of degree at most $3$ which have a degree-$1$ term are the degree-$1$ polynomials, we see that $f(X)$ is a complete mapping over $\F_4$ if and only if \eqref{33} holds. 
If $q=2$ then \eqref{31},\eqref{34}, and \eqref{35} do not hold, and each of \eqref{30} and \eqref{32} implies \eqref{33}.
Thus the result is true when $q=2$, so we assume henceforth that $q>2$.

By Theorem~\ref{cm-bb-deg3}, $f(X)$ is a complete mapping of $\F_{q^2}$ if and only if either \eqref{cm1} or \eqref{cmadd} holds. 
Since \eqref{cmadd} appears in the conclusion of Theorem~\ref{cmthm}, it remains to determine the possibilities for $a,b,c,d$ when \eqref{cm1} holds. By definition, \eqref{cm1} holds if and only if
\[
f(X) \equiv \frac{\lambda X^q-\beta^qX}{\lambda^{q+1}-\beta^{q+1}}\circ \gamma X^{q+2} \circ (\lambda X^q+\beta X) \pmod{X^{q^2}-X}
\]
for some $\lambda,\beta,\gamma\in\F_{q^2}$ such that $\lambda^{q+1}\ne \beta^{q+1}$ and $\gamma^{2q-2}-\gamma^{q-1}+1=0$. Since $q>2$, the monomials $X^{3q}$, $X^{2q+1}$, $X^{q+2}$, and $X^3$ are pairwise incongruent mod $X^{q^2}-X$. 
Thus \eqref{cm1} holds if and only if there exist $\lambda,\beta,\gamma\in\F_{q^2}$ such that all of the following hold:
\begin{enumerate}
\item $\lambda^{q+1}\ne \beta^{q+1}$,
\item $\omega\colonequals -\gamma^{q-1}$ satisfies $\omega^2+\omega+1=0$,
\item $a(\lambda^{q+1}-\beta^{q+1})=\lambda^2 \beta^{2q} (\gamma^q-\gamma)$,
\item $b(\lambda^{q+1}-\beta^{q+1})=\lambda\beta^{2q+1}(\gamma^q-2\gamma)+\lambda^{q+2}\beta^q(2\gamma^q-\gamma)$,
\item $c(\lambda^{q+1}-\beta^{q+1})=2\lambda^{q+1}\beta^{q+1}(\gamma^q-\gamma)+\lambda^{2q+2}\gamma^q-\beta^{2q+2}\gamma$,
\item $d(\lambda^{q+1}-\beta^{q+1})=\lambda^{2q+1}\beta\gamma^q-\lambda^q\beta^{q+2}\gamma$.
\end{enumerate}
We may assume that $q\not\equiv 1\pmod 3$, since this condition follows from (2) and also appears in the conclusion of Theorem~\ref{cmthm}. If $\lambda=0$ and (2) holds then (1)--(6) hold if and only if $\beta\ne 0=a=b=d$ and $c=\beta^{q+1}\gamma$. 
If $\beta=0$ and (2) holds then (1)--(6) hold if and only if $\lambda\ne 0=a=b=d$ and $c=\lambda^{q+1}\gamma^q$. Thus (1)--(6) hold with $\lambda\beta=0$ if and only if \eqref{30} holds.

We now show that \eqref{31} holds if and only if $\gcd(q,6)=1$ and (1)--(6) hold with $\lambda\beta\ne 0$. It is straightforward to check that if $\gcd(q,6)=1$ and (1)--(6) hold with $\lambda\beta\ne 0$ then \eqref{31} holds.
%
%
Conversely, suppose \eqref{31} holds.
Then $24a^2d\ne 0$, so that $\gcd(q,6)=1$.
Pick any $\gamma\in\F_{q^2}^*$ for which $\omega\colonequals -\gamma^{q-1}$ has order $3$.
Then it it routine to verify that
\[
\delta\colonequals \frac{-\omega^2(b+3d^q)\cdot\bigl(b+(1+2\omega)d^q\bigr)}{12a\gamma}
\]
is in $\F_q^*$.
%
%
Let $\beta$ be any element of $\F_{q^2}^*$ such that $\beta^{q+1}=\delta$, and put
\[
\lambda \colonequals \frac{2(1-\omega)a\beta}{b+(1+2\omega)d^q}.
\]
Then $\lambda\in\F_{q^2}^*$, and it is routine to verify that (1)--(6) hold.
%
%

Now suppose that $3\mid q$. Then (1)--(6) imply $b=0$, and also each of \eqref{34} and \eqref{35} implies $b=0$. If (1)--(6) hold with $\lambda,\beta\ne 0$ then $\gamma^{q-1}=-1$, and $d$ equals $0$ if and only if $\lambda^{q+1}=-\beta^{q+1}$, in which case it is easy to check that \eqref{34} holds.
%
%
Conversely, if \eqref{34} holds then $\gamma\colonequals c$ satisfies $\gamma^{q-1}=-1$, so (2) holds. Pick any $\beta\in\F_{q^2}^*$ with $\beta^{q+1}=-1$. 
By hypothesis, $(-a/c)^{(q^2-1)/2}=1$, so that $-a/c$ is a square in $\F_{q^2}^*$ and thus we may choose $\lambda\in\F_{q^2}^*$ with $(\lambda/\beta)^2=-a/c$. Then it is easy to check that (1)--(6) hold.
%
%
If (1)--(6) hold with $\lambda,\beta,d\ne 0$ then it is easy to check that \eqref{35} holds, where $d^{4q+4}+a^4d^{q+5}$ is the square of $\gamma^4 (\lambda\beta)^{2q+2} (\beta^{q+1}+\lambda^{q+1})^3/(\beta^{q+1}-\lambda^{q+1})^3$.
%
%
Conversely, if \eqref{35} holds then put $\gamma\colonequals 1/(ad^2)$, so that (2) holds. Let $\alpha$ be a square root of $d^{4q+4}+a^4 d^{q+5}$, and put $\delta\colonequals c\gamma^{-1} + \alpha$. 
Then $\delta\in\F_q^*$, so we may choose $\lambda\in\F_{q^2}^*$ with $\lambda^{q+1}=\delta$, and then put $\beta\colonequals \lambda(acd^2-\delta-d^{2q+2})/(ad^{q+2})$. It is easy to check that (1)--(6) hold.
%

Finally, suppose $q$ is even. It is routine to verify that (1)--(6) imply \eqref{32}.
%
%
Conversely, suppose \eqref{32} holds. If $b=0$ then one can check that (1)--(6) hold for $\lambda=0$, $\beta=1$, and $\gamma=c$.
%
%
Finally, if $b\ne 0$ then pick any $\omega\in\F_4\setminus\F_2$. 
It is easy to check that (1)--(6) hold for $\gamma^2=\omega c^{2q}+\omega^2 c^{q+1}+c^2$, $\beta=1$, and $\lambda=(c+\gamma)/b^q$. This concludes the proof.
\end{proof}




\end{document}